\documentclass[11pt,reqno,fleqn]{amsart}
\usepackage{amsfonts,amsmath,amsthm,amssymb, enumerate}
\usepackage{pdfsync}

\usepackage{mathrsfs}

\newtheorem{thm}{Theorem}
\newtheorem*{thmA}{Theorem A (Magyar, Stein, and Wainger \cite{MSW})}
\newtheorem*{thmB}{Theorem B (Magyar \cite{Mag})}

\newtheorem{lemma}{Lemma}

\newcommand{\R}{\ensuremath{\mathbb{R}}}
\newcommand{\Z}{\ensuremath{\mathbb{Z}}}
\newcommand{\C}{\ensuremath{\mathbb{C}}}

\newcommand{\1}{\ensuremath{\mathbf{1}}}

\newcommand{\bfa}{\ensuremath{\mathbf{a}}}
\newcommand{\bfq}{\ensuremath{\mathbf{q}}}

\newcommand{\<}{\ensuremath{\lesssim}}
\newcommand{\gs}{\ensuremath{\gtrsim}}

\newcommand{\eps}{\ensuremath{\varepsilon}}
\newcommand{\la}{\ensuremath{\lambda}}

\newcommand{\what}{\ensuremath{\widehat{\omega}_N}}
\newcommand{\ds}{\ensuremath{\widetilde{d\sigma}}}

\newcommand{\dsN}{\ensuremath{\widetilde{d\sigma}_N}}
\newcommand{\dsk}{\ensuremath{\widetilde{d\sigma_k}}}

\newcommand{\eq}{\begin{equation}}
\newcommand{\ee}{\end{equation}}
\newcommand{\p}{\ensuremath{\mathfrak{p}}}

\numberwithin{equation}{section}

\begin{document}
\title[Maximal Function Ineq. and a Thm of Birch]{Maximal Function Inequalities and a Theorem of Birch}
\author{Brian Cook}

\thanks{
The author was supported by the Fields Institute and by NSF grant DMS1147523}

\begin{abstract} In this paper we prove an analogue of the discrete spherical maximal theorem of Magyar, Stein, and Wainger, an analogue which concerns maximal functions associated to homogenous algebraic hypersurfaces. Let $\p$ be a homogenous polynomial in $n$ variables with integer coefficients of degree $d>1$. The maximal functions we consider are defined by \[
A_*f(y)=\sup_{N\geq1}\left|\frac{1}{r(N)}\sum_{\p(x)=0;\,x\in[N]^n}f(y-x)\right|\]
for functions $f:\Z^n\to\mathbb{C}$, where $[N]=\{-N,-N+1,...,N\}$ and $r(N)$ represents the number of integral points on the surface defined by $\p(x)=0$ inside the $n$-cube $[N]^n.$
It is shown here that the operators $A_*$ are bounded on $\ell^p$ in the optimal range $p>1$ under certain regularity assumptions on the polynomial $\p$.
\end{abstract}

\maketitle

\section{Introduction}
\subsection{Results}
  In \cite{MSW}  Magyar, Stein, and Wainger provided a number theoretic analogue to Stein's well known spherical maximal theorem on $\R^n$. Let $|x|^2=x_1^2+...+x_n^2$ for $x\in \Z^n$ and for a fixed integer $\la>0$ define the operators 
\[
S_\la f(y)=\frac{1}{r(\la)}\sum_{|x|^2=\la} f(y-x)
\]
for $f:\Z^n\to\mathbb{C}$. Here $r(\la)$ is simply the number of representations of $\lambda$ as a sum of $n$ squares of integers. Of interest is the maximal function given by
\[
S_* f(y)=\sup_{\la\geq1}|S_\la f(y)|.
\]
For a function $f$ defined on $\Z^n$ we use the notation $||f||_{\ell^p}$ to denote the  norm $
\left(\sum_{x\in \Z^n}|f(x)|^p\right)^{1/p}$. 

\begin{thmA}\label{thmA}
Let $p> 1$ be a fixed real number. There is a constant $C$ such that 
\[
||S_*f||_{\ell_p}\leq C||f||_{\ell^p},
\]
for all $f\in \ell^p$, i.e. $S_*$ is bounded on $\ell^p$, if and only if $n\geq 5$ and  $p>n/(n-2)$.
\end{thmA}

An extension of this result to certain algebraic hypersurfaces of higher degree is  given by Magyar in \cite{Mag}. Let $\p$ be an integral form, i.e., a homogenous polynomial with integral coefficients, in $n$ variables of degree $d>1$. If $\p$ is positive then one can ask a similar question regarding averages over  integral points on the family of surfaces defined by $\p(x)=\la$. Approaching this question requires a knowledge of the set of integral points, and provided that the quantity 
\[
\mathcal{B}(\p)=codim\{z\in \C^n:\partial_{z_1}\p(z)=...=\partial_{z_n}\p(z)=0,\}
\]
(known as the Birch rank) is strictly greater than $(d-1)2^d$ this information is provided by Birch in \cite{Bi}. In particular, one sees that there exists an infinite arithmetic progression $\Gamma_\p$ and nonnegative constants $C,C'$ such that 
\[r(\la)=\sum_{\p(x)=\la} 1\]
satisfies $C\la^{(n/d)-1}\leq r(\la) \leq C'\la^{(n/d)-1}$ for all $\la\in \Gamma_\p$. One then defines the operators
\[
T_\la f(y)=\frac{1}{r(\la)}\sum_{\p(x)=\la} f(y-x)
\]
for $\la\in \Gamma_\p$, and   the maximal  function
\[
T_* f(y)=\sup_{\la\in\Gamma_\p}|T_\la f(y)|.
\]

\begin{thmB}\label{thmB}
If $\p$ is a positive form of degree $d>1$ with $\mathcal{B}(\p)>(d-1)2^d$ then there is a constant $C$ such that 
\[
||T_*f||_{\ell_2}\leq C||f||_{\ell^2}
\]
for all $f\in \ell^2$.
\end{thmB}

The range $p> n/(n-2)$ can be obtained in this result for the case of positive definite quadratic forms in at least five variables. In general the expected sharp range is $p>n/(n-d)$, but for $d>2$ this seems a difficult question.

In this paper we are interested in obtaining similar types of results for a slightly different collection of discrete maximal operators.  Let $\p$ again be an integral form of degree $d>1$ in $n$ variables and define the convolution operators  \[
A_Nf(y)=\frac{1}{r(N)}\sum_{x\in[N]^n;\,\p(x)=0}f(y-x)\]
and the associated discrete maximal operators \[
A_*f(y)=\sup_{N\geq1} |A_Nf(y)|.\]
Analogously the normalization factor is defined as \[
r(N)=\sum_{x\in[N]^n;\,\p(x)=0}1\,\,.\]
We  work under a large rank condition on $\p$ which forces $r(N)\< N^{n-d}$. One can guarantee also that $r(N)\gs N^{n-d}$ by assuming that $\p(x)=0$ has a nonsingular solutions in every $p$-adic completion of $\mathbb{Q}$ (including $\mathbb{Q}_\infty=\R$).  Forms of degree $d>1$ with such nonsingular solutions and $\mathcal{B}(\p)>(d-1)2^d$ will be called $regular$ and we restrict our attention to such forms.

Our main result is the following.

\begin{thm}\label{thm11}
If $\p$  is a regular form of degree $d>1$ then  $A_*$ is bounded on $\ell^p$ if and only if $p>1$. 
\end{thm}

The reader should note that the condition on the nonsingular real solution in our assumptions rules out positive polynomials, so this result is indeed disjoint from Theorem B. The method of \cite{Mag} does however extend to give a proof for the $\ell^2$ case of Theorem 1. Similarly, the methods of Magyar, Stein, and Wainger apply to give the result for indefinite quadratic forms of rank at least $5$ in the range $p>n/(n-2)$. Also note that the results of \cite{MT} cover  special cases of our main result in the full range $p>1$.

 To see the `only if' requirement in the statement of  Theorem \ref{thm11} we simply need to show by example that the result fails when $p=1$. For this we can consider precisely the same example that is used to prove that the range of $p$ is optimal in Theorem A, which is an insight attributed to A.  Ionescu.  Let $f_0$ be the function which is one at the origin and is otherwise zero. For  $f_0$ we have \[A_*f_0(y)\approx\1_{V}(y)||y||^{n-d}_{\ell^\infty}\] where $\1_V$ is the characteristic function of the set $V=\{x\in\Z^n:\p(x)=0\}$ and \[||y||_{\ell^\infty}=\sup_{i=1,...,n}|y_i|.\] For a surface with $r(N)\gs N^{n-d}$, as is the case for regular $\p$, it is easy to see that $A_*f_0$ is not in $\ell^1$.

\subsection{Overview}

A worthwhile exercise for us at this point is to identify the key steps used in the proof of Theorem A.  Several features of our approach are similar, and this helps  highlight some relevant differences  between the operators considered there and the ones treated below. The outline goes as follows. \\

i): Approximate the Fourier multipliers of the $S_\la$ with the circle method. The multipliers are given by \[\widehat{\sigma_\la}(\xi)=\frac{1}{r(\la)}\sum_{|x|^2=\la}e(x\cdot\xi).\] One gets a decomposition of the form    \[
\widehat{\sigma_\la}=m_\la+e_\la\]
where $m_\la$ takes the shape as a sum of `major arc' terms, \eq\label{1.21}
\sum_{q=1}^\infty\sum_{a\in U_q}e(-\la a/q)\,m^{a/q}_\la(\xi),\ee
and $|e_\la|\< \la^{-\delta}$ uniformly. The operators $M_\la$, $M_\la^{a/q}$ and $E_\la$ are then defined by the multipliers $m_\la$, $m_\la^{a/q}$, and $e_\la$, respectively, giving a decomposition of the spherical operator $S_\la$ as  
\[S_\la =M_\la + E_\la  =\sum_{q=1}^\infty\sum_{a\in U_q}e(-\la a/q)M_\la^{a/q}+E_\la .\] 
One in turn defines the maximal operators $M_*f=\sup_\la|M_\la f|$,  $M_*^{a/q}f=\sup_\la|M_\la^{a/q}f|$, and $E_*f=\sup_\la|E_\la f|$.

ii): The associated maximal operators $M_*^{a/q}$ satisfy the estimate \eq\label{1.22}
||\sup_{\lambda}|M^{a/q}_\la f|\,||_{\ell^p}\<q^{-n(1-1/p)+\eps}||f||_{\ell^p}\ee
for each $q$ (uniformly in $a$). This involves a reduction to the spherical maximal theorem on $\R^n$. Then an application of  the triangle inequality gives \[
|| M_*||_{\ell^p\to\ell^p}\leq\sum_{q=1}^\infty \sum_{a\in U_q} || M^{a/q}_*||_{\ell^p\to\ell^p}\< 1\] 
when $p>n/(n-2)$.

iii):  From \cite{Mag0} we have the partial maximal function inequality\[
||\sup_{\lambda_0\leq\lambda<2\lambda_0}|S_\la f|\,||_{\ell^p}\<||f||_{\ell^p}\]
for all $p>n/(n-2)$.

iv): With the uniform estimates in i), the operators $E_\la$ satisfy the estimate\[
||\sup_{\lambda_0\leq\lambda<2\lambda_0}|E_\la f|\,||_{\ell^2}\<\lambda_0^{-\delta}||f||_{\ell^2}.\]
 This is turn implies that \[
 ||\sup_{\lambda_0\leq\lambda<2\lambda_0}|E_\la f|\,||_{\ell^p}\<\lambda_0^{-\delta}||f||_{\ell^p}\]
for all $p>n/(n-2)$ by  interpolation simply by noticing  that  $E_\la=S_\la-M_\la$ and applying the estimates obtained for $S_*$ and $M_*$ in parts ii) and iii). Finally 
 \[
 ||\sup_{\lambda}|E_\la f|\,||_{\ell^p}\leq\sum_{j=0}^\infty||\sup_{2^j\leq\lambda<2^{j+1}}|E_\la f|\,||_{\ell^p}\<\sum_{j=0}^\infty2^{-\delta j}||f||_{\ell^p}\<||f||_{\ell^p}\]
when $p>n/(n-2)$.\\

Ideally we would like to follow this outline also, but we run into a problem with the estimate in \eqref{1.22}. In our situation the analogous $M^{a/q}_*$ are morally identical, and hence there is no real hope that the estimates  can be strengthened to the point where we can simply sum the individual $\ell^p\to\ell^p$ norms over $q$. This means, for example, in the case of quadratic forms this outline can never achieve inequalities for $p$ below $n/(n-2)$. 

On the positive side for us, though, is that the analogue of \eqref{1.21} doesn't contain the character $e(-\la a/q)$. This translates to the fact that we are not obliged to apply the triangle inequality when obtaining the  $\ell^p\to\ell^p$ estimate for $M_*$. In turn this opens up  the possibility of dealing directly with $||M_* f||_{\ell^p}$ without partitioning $M_*$ into its `major arc' constituents. At this point we can take inspiration from Bourgain, as the difficulties he conquers in \cite{Bou}  are similar in nature. 

The overall argument presented for  Theorem \ref{thm11} is indeed an amalgamation of the methods used in the works of Bourgain \cite{Bou}, Magyar \cite{Mag}, and Magyar, Stein, and Wainger \cite{MSW}. The model employed here shares similarities with the outline above, the main differences being that there are further modifications to the main terms obtained in step i) and  steps ii) and iv) are run more concurrently. The generality of their works allows us to modify certain aspects in a relatively straightforward manner (for example, obtaining the relevant partial maximal function inequality), and in some cases (for example, obtaining initial approximations for the multipliers)  we can borrow results directly. That being said, carrying this out is not a straightforward application of what is done previously. Here the new insight boils down to a finer analysis of exponential sums, as we require a class of sums which is more general than those used in \cite{Mag}, something which is motivated by \cite{GWpaper}.

The  paper is formatted as follows. In section 2 we formulate several auxiliary results needed in the the proof of the  main result, and the proof subject to these results is given in section 3. The remainder of the paper is comprised of sections dedicated to proving the auxiliary results formulated in  section 2,  as well as handling  a few other necessary items that we shall need. In section 4 we devote ourselves to  results concerning exponential sums, and in section 5 we reconsider  the initial approximation for the Fourier multipliers of the operators $A_N$. Sections 6 and 7 respectively contain certain results related to $\ell^p$ estimates ($p<2$) estimates and $\ell^2$ estimates. Finally, section 8 gives a proof of the  partial maximal operators estimate that is needed.

%%%%%%%%%%%%%%%%%%%%%%%%%%%%%%%%%%
\vspace{.2in}
\section{Preliminary results}

There is a somewhat lengthy list of results presented here which are organized into subsections; one  for exponential sum results, one on Fourier multiplier approximations,  one on a continuous maximal function estimate, and a further subsection on results related to the approximations. The one thing missing in this breakdown is the partial maximal function estimate which we  present now. 
 
\begin{lemma}\label{lemma2.9}
For fixed  $p>1$ we have that 
\[
|| \sup_{N_0\leq N\leq N_0^2}A_N||_{\ell^p\to\ell^p}=O(\log\log(N_0)).
\]
whenever $\p$ is a regular form. 
\end{lemma}

\bigskip
%%%%%%%%%%%%%%%%%%%%%%%%%
\subsection{Exponential sums}

The exponential sums we are working with are defined as \[
F(a,q,\bfa,\bfq)=Q^{-n}\sum_{s\in Z^n_Q}e(\p(s)a/q+s\cdot\bfa/\bfq),\]
where $\bfq=(q_1,...,q_n)$, $\bfa/\bfq=(a_1/q_1,...,a_n/q_n)$, and $Q=lcm(q,q_1,...,q_n)$ is the least common multiple. 
For a given positive integer $q$ the notation $Z_q$ is used for the cyclic group $\Z/q\Z$, and $U_q$ denotes the multiplicative group $Z^*_q$. We set $U_1=Z_1$ the group consisting of the single element $0$.  When necessary, the group $Z_q$ ($U_q$) should be identified as (in) the set $\{0,1,...,q-1\}\subset\Z.$

A important observation is that in several cases there is sufficient cancellation to give $F(a,q,\bfa,\bfq)=0$. 

\begin{lemma}\label{lemma2.1}
Let $q\geq 1$ be a given integer. If $ q_i\not|\,q$ for some $1\leq i\leq n$ then for any fixed $a\in U_q$ and $a_i\in U_{q_i}$, $i=1,...,n$, we have
 \[F(a,q,\bfa,\bfq)=0.\]
\end{lemma}

The exponential sums  that appear in \cite{Mag} present themselves as a special case of the the $F(a,q,\bfa,\bfq)$. This is the case when $q=q_1=...=q_n$, and hence $Q=q$. In this situation we use a slightly different notation for convenience:\[
F_q(a,\bfa)=q^{-n}\sum_{s\in Z_q^n}e(\p(s)a/q +s\cdot \bfa/q).\]

From \cite{Bi} we inherit the following estimate.
\begin{lemma}\label{lemma2.2}
For $q\geq1$ we have  
\[
F_q(a,\bfa)=O(q^{-c}).
\]
for all $a\in U_q$. Here $c=\mathcal{B}(\p)(d-1)^{-1}2^{1-d}$. This estimate is uniform for  $\bfa\in Z_q^n$.
\end{lemma}

The $O$ notation is used in the normal way, and we also use the alternate notation $f\<g$ frequently to replace $f=O(g)$. The implied constants throughout are allowed to depend on any parameter which is not $q$ or one associated directly to $N$ (such as our later use of $k$).

One final thing that we wish to notice is a simple but useful observation based on the identity \[
\sum_{a\in Z_q}g(a/q) =\sum_{d|q}\sum_{a\in U_d}g(a/d),\]
where the sum in $d$ is overall all divisors of $q$. 
\begin{lemma}\label{lemma2.3}
For $q\geq 1$ given we have
\[
\sum_{q_1|q}...\sum_{q_n|q}\sum_{a_1\in U_{q_1}}...\sum_{a_n\in U_{q_n}}F(a,q,\bfa,\bfq)g(\bfa/\bfq)=\sum_{\bfa\in Z_q^n}F_q(a,\bfa)g(\bfa/q)
\]
for any function $g$ which is defined on the set $\{(a_1/q,...a_n/q):0\leq a_1,...,a_n\leq q-1\}$.
\end{lemma}

\bigskip
%%%%%%%%%%%%%%%%%%%%
\subsection{Approximations I}
Here we consider some initial approximations for the associated Fourier multipliers. The multipliers are given by the normalized discrete Fourier transforms \[
\what(\xi) =\frac{1}{r(N)}\sum_{x\in[N]^n;\,\p(x)=0}e(x\cdot \xi),\]
 where $e(z)=e^{2\pi i z}$ and the notation $\widehat{f}$ denotes the Fourier transform for functions on $\Z^n$:\[
\widehat{f}(\xi)= \sum_{x\in\Z^n}f(x)e(x\cdot \xi).\]
The notation $[N]$ is shorthand for $\{-N,-N+1,..., N\}$. We ultimately  borrow an initial approximation for the functions $\what$ from \cite{Mag}, although one should note that in this work $\what$ is viewed as a function on the $n$-torus $\Pi^n=(\R/\Z)^n$, where the torus $\Pi$ is identified with the real interval $[-1/2,1/2]$ (with endpoints identified) and is equipped with the Lebesgue measure. 

The next result is essentially  (\cite{Mag}, Lemma 1), the proof being identical. 
 \begin{lemma}\label{lemma2.4}
Let $\p$ be a regular form. There exists a constant $\kappa>0$ and a  $\delta>0$ such that 
\[\what(\xi)=\kappa N^{d-n}\sum_{q=1}^\infty\sum_{a\in U_q}\sum_{\bfa\in Z_q^n} F_q(a,\bfa)\zeta(q(\xi-\bfa/q))\ds_N(\xi-\bfa/q)+O(N^{-\delta}).\]
\end{lemma}
Here (and always)  $\delta$ represents a small positive number, not necessarily the same at each occurrence, and $\zeta$ is a fixed smooth bump function equal to one on $[-1/10,1/10]^n$ and supported on $[-1/5,1/5]^n$. The term $\dsN$ represents the Fourier transform of a measure supported on the surface given by $\p(x)=0$, something  which is discussed in more detail the next subsection.

 Denote the  dyadic interval of integers $[2^l,2^{l+1})$ by $I_l$  for $l\geq0$. For a fixed $l$ we  define 
\[
M_{N,l}(\xi)=\kappa N^{d-n}\sum_{q\in I_l}\sum_{a\in U_q}\sum_{\bfa\in Z_q^n} F_q(a,\bfa)\zeta(10^l(\xi-\bfa/q))\dsN(\xi-\bfa/q).\]
 We now let $M$ denote a multiplier as opposed to the associated operator in contrast to the discussion in section 1.2.

\begin{lemma}\label{lemma2.6}
If $\p$ is a regular form then there is a $\delta>0$ such that 
\[\what(\xi)=\sum_{l=0}^\infty M_{N,l}(\xi)+O(N^{-\delta}).\]
uniformly in $\xi$.\end{lemma}

The form of the approximation in Lemma \ref{lemma2.4} requires modifications. This lemma, the first of the changes, follows from applications of  Lemma \ref{lemma2.2} and Lemma \ref{lemma2.5} (which is a Fourier decay estimate for $\ds$ stated below in section 2.3).

We also have the following maximal function estimate for the $M_{N,l}$. The particular phrasing in this result  is useful later on.
\begin{lemma}\label{lemma2.115}
Let $j\geq1$. Then
\[
||\sup_{N=2^k;\,k\geq 4^{j-1}}  |\mathscr{F}^{-1}(M_{N,j}\widehat{f})|\,||_{\ell^2}\< 2^{-\delta j}||f||_{\ell^2}\]
for regular forms $\p$.
\end{lemma}

\bigskip
%%%%%%%%%%%%%%%%%%%%
\subsection{A continuous maximal function estimate}

We begin with a discussion of the terms $\dsN$ appearing in the previous subsection. For a function $f\in L^1(\R^n)$ we use the notation $\widetilde{f}$ to denote the Fourier transform of $f$ over $\R^n$, by which we mean the unique function satisfying \[
f(x)=\int_{\R^n} \widetilde{f}(\xi)e(-x\cdot\xi)d\xi.\]
Define the measure $\sigma_N$ by \[
d\sigma_N(x)=\phi(x/N)\frac{d\mu(x)}{|\nabla\p(x)|}\]
where $\phi$ is a smooth bump function supported on $[-2,2]$ and identically one on $[-1,1]$, and $d\mu$ is the Euclidean surface measure on the surface in $\R^n$ defined by $\p(x)=0$. Birch (\cite{Bi}, Section 6) gives a thorough treatment of related integrals which, in particular, shows that  $\phi(x/N)|\nabla\p(x)|^{-1}$ is an $L^1(d\mu)$ function  when $\p$ is a regular form. Thus  $\sigma_N$  is a  measure supported on the surface patch $V_N=\{x\in\R^n:x\in[-2N,2N]^n,\,\p(x)=0\}$ which is absolutely continuous with respect to $\mu$.
Moreover, under the assumption that $\p$ has a nonsingular real solution in $V_1$ it follows that these measures are positive. 

These types of measures are treated  in (\cite{Mag}, Section 1), where the main analysis there is based on exponential sum estimates from \cite{Bi}.  From here we gain a important insight, namely that we have the representation \[
\ds_N(\xi)=\int_{\R}\int_{\R^n}\phi(x/N)e\left(\p(x)t+x\cdot\xi\right)\,dx\,dt.\]
Basic manipulations give the scaling property \[\ds_1(N\xi)=N^{d-n}\ds_{N}(\xi),\]
which motivates us to define  $d\sigma=\kappa \,d\sigma_1$ where $\kappa$ is the constant introduced in Lemma \ref{lemma2.4}.

We have the following decay estimate for $\ds$, which is proven in (\cite{Mag}, Section 1).

\begin{lemma}\label{lemma2.5}
Assume that $\p(x)=0$ has a nonsingular real solution in $V_1$. Then
\[
\ds(\xi)=O\left(\frac{1}{(1+|\xi|)^{c}}\right)
\]
in $\R^n$, where $c=\mathcal{B}(\p)(d-1)^{-1}2^{1-d}-1$.
\end{lemma}

An  important point to observe at the moment is that  $c$ in both Lemma \ref{lemma2.2} and Lemma \ref{lemma2.5} is strictly greater than $2$ when dealing with regular forms.

As in the proof of Theorem A, part of the argument relies on a comparison a with real variable maximal function analogue. For our purposes we define the continuous convolution operators by
\[
R_Nf(y)=\int_{\R^n}\widetilde{f}(\xi)\ds(N\xi)e(-y\cdot\xi)\,d\xi,\]
for suitable functions $f$ defined on $\R^n$, and the associated maximal operators \[
R_*f(y)=\sup_{N\geq1}|R_Nf(y)|.\]
As an application of Lemma \ref{lemma2.5} we achieve $L^p(\R^n)\to L^p(\R^n)$ estimates for $R_*$. Indeed, as we may write \[
R_Nf(y)=N^{d-n}\int f(y-x) \,d\sigma_N(x)\]
up to constants, it is easy to see that we only need to consider the supremum over the set of dyadic integers. This is something which is of course  true for $A_*$ as well, which is discussed below. A direct application of (\cite{Rubio}, Theorem A) gives a continuous maximal function inequality.

\begin{lemma}\label{lemma2.7}
If $\p$ is a regular form then $R_*$ is a bounded on $L^{p}(\R^n)$ for all $p>1$.
\end{lemma}

\bigskip
\subsection{Approximations II}

Here we  look at some further modifications to the approximations of the $\what$ that are going to be needed. Define the terms 
\[
\Omega_{N, q}(\xi)=\sum_{a\in Z_q}\sum_{\bfa\in Z_q^n}F_q(a,\bfa)\zeta(q^{2}(\xi-\bfa/q))\ds(N(\xi-\bfa/q)).
\]
The purpose for introducing  these terms  is somewhat twofold. The first observation is that these terms are better suited for  $\ell^p$  results when $p<2$, something which manifests itself in the next result. Here we begin our use of the notation $\mathscr{F}^{-1}$  for the inverse Fourier transform. 

\begin{lemma}\label{lemma2.8}
For $q$ fixed and $p>1$ we have
\[
|| \sup_{N\geq 1}|\mathscr{F}^{-1}(\Omega_{N, q}\widehat{f})|\, ||_{\ell^p}\<||f||_{\ell^p}.
\]
The implied constant is independent of $q$. 
\end{lemma}

The other observation observation about  the  $\Omega_{N,q}$  is as follows. Set $Q_j=2^j!$ and notice that \[
\Omega_{N, Q_j}(\xi)=\sum_{a\in Z_{Q_j}}\sum_{\bfa\in Z_{Q_j}^n}F_{Q_j}(a,\bfa)\zeta(Q_j^2(\xi-\bfa/Q_j))\ds(N(\xi-\bfa/Q_j))\]
\[
=\sum_{q|Q_j}\sum_{a\in U_q}\sum_{q_1,...,q_n|Q_j}\sum_{\bfa\in U_{\bfq}}F(a,q,\bfa,\bfq)\zeta(Q_j^2(\xi-\bfa/\bfq))\ds(N(\xi-\bfa/\bfq))\]
\[
=\sum_{q|Q_j}\sum_{a\in U_q}\sum_{q_1,...,q_n|q}\sum_{\bfa\in U_{\bfq}}F(a,q,\bfa,\bfq)\zeta(Q_j^2(\xi-\bfa/\bfq))\ds(N(\xi-\bfa/\bfq))\]
\[
=\sum_{q|Q_j}\sum_{a\in U_q}\sum_{\bfa\in Z^n_{q}}F_q(a,\bfa)\zeta(Q_j^2(\xi-\bfa/q))\ds(N(\xi-\bfa/q)),\]
where we have made use of Lemma \ref{lemma2.1} and the observation of Lemma \ref{lemma2.3} while using the notation $U_\bfq$ to denote $U_{q_1}\times...\times U_{q_n}$. The last line is equal to \[
\sum_{l=0}^{j-1}\sum_{q\in I_l}\sum_{a\in U_q}\sum_{\bfa\in Z^n_q}F_q(a,\bfa)\zeta(Q_j^2(\xi-\bfa/q))\ds(N(\xi-\bfa/q))\]\[
+\sum_{q|Q_j;\,q\geq 2^j}\sum_{a\in U_q}\sum_{\bfa\in Z^n_q}F_q(a,\bfa)\zeta(Q_j^2(\xi-\bfa/q))\ds(N(\xi-\bfa/q)).\]
This motivates the decomposition \[
\Omega_{N, Q_j}=\sum_{l=0}^{j-1}M_{N,l}+E^{(1)}_{N,j}+E^{(2)}_{N,j}\]
where \[
E^{(1)}_{N,j}=\sum_{l=0}^{j-1}\sum_{q\in I_l}\sum_{a\in U_q}\sum_{\bfa\in Z^n_q}F_q(a,\bfa)\left(\zeta(Q_j^2(\xi-\bfa/q))-\zeta(10^l(\xi-\bfa/q))\right)\ds(N(\xi-\bfa/q))\]\
and \[
E^{(2)}_{N,j}=\sum_{q|Q_j;\,q\geq 2^j}\sum_{a\in U_q}\sum_{\bfa\in Z^n_q}F_q(a,\bfa)\zeta(Q_j^2(\xi-\bfa/q))\ds(N(\xi-\bfa/q)).\]

Lemma \ref{lemma2.8} can now be viewed as a method of providing an estimates which  simultaneously controls a large collection of the `major arc'  terms. This of course relies on our ability to adequately control the $E^{(i)}_{N,j}$, forming our second observation about the $\Omega$ terms. That this is indeed the case for the error terms forms the content of the next two results. 

\begin{lemma}\label{lemma2.10} If $\p$ a is regular form then 
\[
||\sup_{N=2^k;\,k\geq 4^{j-1}} |\mathscr{F}^{-1}(E^{(1)}_{N,j}\widehat{f})|\,||_{\ell^2}\< 2^{-\delta j}.\]
for all $j\geq 1$.\end{lemma}

\begin{lemma}\label{lemma2.11}
We have 
\[
||\sup_{N=2^k;\,k\geq 4^{j-1}}  |\mathscr{F}^{-1}(E^{(2)}_{N,j}\widehat{f})|\,||_{\ell^2}\< 2^{-\delta j}\]
for regular forms $\p$ when $j\geq1$.
\end{lemma}

The final result presented here is another $\ell^2$ estimate. 
\begin{lemma}\label{lemma2.12}
We have the estimate
\[
||\sup_{k\geq 4^{j-1}}|\omega_{2^k}*f-\mathscr{F}^{-1}(\sum_{l=0}^{j-1}M_{2^k,l}\widehat{f})|\,||_{\ell^2}\< 2^{-\delta j}\]
when $\p$ is a regular from.
\end{lemma}

%%%%%%%%%%%%%%%%%%%%%%%%%%%%%%%%%%%%
\vspace{.2in}
\section{Proof of Theorem \ref{thm11}}

A basic observation is that we only  need to consider the supremum over $N$ of the form $2^k$, $k\geq1$, as when $f\geq0$ we see that\[
\sup_{2^k\leq N<2^{k+1}}A_Nf\leq \frac{1}{r(2^k)}\sum_{x\in[2^{k+1}]^n}f(y-x)=\frac{r(2^{k+1})}{r(2^k)}A_{2N}f\]
and $r(2^{k+1})/r(2^k)\<1$ independent of $k$.

To proceed, we let \[
Q_j=2^j!\] 
as above and  set \[
H_j=[4^{j-1},4^{j})\]
 for $j\geq 1$.  We slightly alter previous notations for simplicity:
 \[
\Omega_{k, j}(\xi):=\Omega_{2^k,Q_j}(\xi)=\sum_{a\in Z_{Q_j}}\sum_{\bfa\in Z_{Q_j}^n}F_{Q_j}(a,\bfa)\zeta(Q_j^2(\xi-\bfa/Q_j))\dsk(\xi-\bfa/Q_j)
\]
for $j\leq j_0$, $k\in H_{j_0}$, where $\dsk(\xi)$ is used to denote $\ds(2^k\xi)$. Also tet
$K_{k}$ be the convolution kernel for $A_{2^k}$, i.e. $K_k=\omega_{2^k}$ and $A_{2^k}f=\omega_{2^f}f$ is the operator appearing in Theorem \ref{thm11}.

When $k\in H_{j_0}$ we can write
\[
K_{k}*f=\mathscr{F}^{-1}(\Omega_{k, 1}\widehat{f})+\mathscr{F}^{-1}((\Omega_{k, 2}-\Omega_{k, 1})\widehat{f})
+...+(K_k*f-\mathscr{F}^{-1}(\Omega_{k, j_0}\widehat{f}))\]
which gives \eq\label{3.0}
\sup_{k\geq 1}|K_k*f|\leq \sum_{j=1}^\infty\sup_{k\geq 4^{j-1}}|\mathscr{F}^{-1}((\Omega_{k, j}-\Omega_{k, j-1})\widehat{f}))|+\sum_{j_0=1}^\infty\sup_{k\in H_{j_0}}|K_k*f -\mathscr{F}^{-1}(\Omega_{k, j_0}\widehat{f})|,\ee
where it is to be understood that $\Omega_{k,0}=0$. Take the $\ell^p$ norm of \eqref{3.0} and apply the triangle inequality. Then, for all $p>1$, we need to obtain estimates for \[
||\sup_{k\geq4^{j-1}}|\mathscr{F}^{-1}(\Omega_{k, j}-\Omega_{k, j-1})\widehat{f})| ||_{\ell^p}\]
and for \[
||\sup_{k\in H_{j_0}}|K_k*f -\mathscr{F}^{-1}(\Omega_{k, j_0}\widehat{f})|\,||_{\ell^p}.\]
For the former, we see that \eq\label{3.3}
||\sup_{k\geq4^{j-1}}|\mathscr{F}^{-1}(\Omega_{k, j}-\Omega_{k, j-1})\widehat{f})| ||_{\ell^p}\< ||f||_{\ell^p}\ee
for all $j>0$ by the triangle inequality and Lemma \ref{lemma2.8}. For the latter we have that \eq\label{3.4}
||\sup_{k\in H_{j_0}}|K_k*f -\mathscr{F}^{-1}(\Omega_{k, j_0}\widehat{f})|\,||_{\ell^p}\< j_0||f||_{\ell^p}\ee
for all $p>1$. This estimate  follows  by the triangle inequality, Lemma \ref{lemma2.8}, and two applications of Lemma \ref{lemma2.9}, as\[
||\sup_{2^{4^{j_0-1}}\leq N<2^{4^{j_0}}}|A_Nf|\,||_{\ell^p}\leq ||\sup_{2^{4^{j_0-1}}\leq N<2^{2\cdot4^{j_0-1}}}|A_Nf|\,||_{\ell^p}+||\sup_{2^{2\cdot4^{j_0-1}}\leq N<2^{4^{j_0}}}|A_Nf|\,||_{\ell^p}\< j_0.\]

Now \eqref{3.3} and \eqref{3.4} need to be interpolated against stronger $\ell^2$ estimates. Provided the estimates at $\ell^2$ are of the form $2^{-\delta j}$ and $2^{-\delta j_0}$ (resp.), it follows that interpolation between the $\ell^2$ estimates and the $\ell^{p_0}$ estimates, for any fixed $p_0>1$, on the left hand sides of \eqref{3.3}  and \eqref{3.4} yield terms that are summable in $j$  and $j_0$ (resp.) and thus proving  Theorem \ref{thm11}.

We consider \eq\label{3.5}
||\sup_{k\geq 4^{j-1}}|K_k*f-\mathscr{F}^{-1}(\Omega_{k,j}\widehat{f})|\,||_{\ell^2}.\ee
This can be estimated by \[
||\sup_{k\geq 4^{j-1}}|K_k*f-\mathscr{F}^{-1}(\sum_{l=0}^{j-1}M_{2^k,l}\widehat{f})|\,||_{\ell^2}
+||\sup_{k\geq 4^{j-1}}|\mathscr{F}^{-1}((\sum_{l=0}^{j-1}M_{2^k,l}-\Omega_{k,j})\widehat{f})|\,||_{\ell^2}\]\[
=||\sup_{k\geq 4^{j-1}}|K_k*f-\mathscr{F}^{-1}(\sum_{l=0}^{j-1}M_{2^k,l}\widehat{f})|\,||_{\ell^2}
+||\sup_{k\geq 4^{j-1}}|\mathscr{F}^{-1}(E^{(1)}_{k,j}\widehat{f})|\,||_{\ell^2}+||\sup_{k\geq 4^{j-1}}|\mathscr{F}^{-1}(E^{(2)}_{k,j})\widehat{f})|\,||_{\ell^2}.\]
Each of these terms is $O(2^{-\delta j})$ by Lemmas \ref{lemma2.10}, \ref{lemma2.11}, and \ref{lemma2.12}.

To finish the argument we estimate \[
||\sup_{k\geq 4^{j-1}}|\mathscr{F}^{-1}((\Omega_{k,j}-\Omega_{k,j-1})\widehat{f})|\,||_{\ell^2}\]
by observing that \[
\Omega_{k,j}-\Omega_{k,j-1}=M_{2^k,j-1}+E^{(1)}_{k,j}-E^{(1)}_{k,j-1}+E^{(2)}_{k,j}-E^{(2)}_{k,j-1}.\]
The terms arising from the $E^{(i)}_{k,j}$ can be treated by applying Lemmas \ref{lemma2.10} and \ref{lemma2.11}. The remaining term arising from $M_{2^k,j}$ is treated by Lemma \ref{lemma2.115}. Summing these estimates  gives a bound of the form $O(2^{-\delta j})$. Finally notice that \[
||\sup_{k\in H_{j_0}}|K_k*f -\mathscr{F}^{-1}(\Omega_{k, j_0}\widehat{f})|\,||_{\ell^p}\leq||\sup_{k\geq 4^{j_0-1}}|K_k*f -\mathscr{F}^{-1}(\Omega_{k, j_0}\widehat{f})|\,||_{\ell^p},\]
so estimate \eqref{3.5} completes the proof.

\vspace{.2in}
%%%%%%%%%%%%%%%%%%%%%%%%%%%%
\section{Exponential sums}

In this section  we  provide proofs of  the results related to  exponential sum statements presented in section 2.1. First we consider the observation of Lemma \ref{lemma2.3}, and then proceed to the proof of  Lemma \ref{lemma2.1}. There is also another previously unstated result  treated here that is  needed in the proof of Lemma \ref{lemma2.8}. 

For Lemma \ref{lemma2.3} we take a function $g$ as stated and then \[
\sum_{\bfa\in Z_q^n}F_q(a,\bfa)g(\bfa/q)=q^{-n}\sum_{s\in Z^q}e(\p(s)a/q)\sum_{\bfa\in Z_q^n}e(\bfa\cdot s/q)g(\bfa/q)
\]\[
=q^{-n}\sum_{s\in Z_q^n}e(\p(s)a/q)\sum_{q_1|q}...\sum_{q_n|q}\sum_{a_1\in U_{q_1}}...\sum_{a_n\in U_{q_n}}e(a_1\cdot s_1/q_1)...e(a_n\cdot s_n/q_n)g(a_q/q_1,...,a_n/q_n)\]\[
=q^{-n}\sum_{s\in Z_q^n}e(\p(s)a/q)\sum_{q_1|q}...\sum_{q_n|q}\sum_{\bfa\in U_\bfq}e(s\cdot \bfa/\bfq)g(\bfa/\bfq)\]\[
=\sum_{q_1|q}...\sum_{q_n|q}\sum_{\bfa\in U_\bfq}q^{-n}\sum_{s\in Z_q^n}e\left(\p(s)a/q+s\cdot \bfa/\bfq)g(\bfa/\bfq\right)\]
\[
=\sum_{q_1|q}...\sum_{q_n|q}\sum_{\bfa\in U_\bfq}F(a,q,\bfa,\bfq)g(\bfa/\bfq),\]
 noting that $lcm(q,q_1,...,q_n)$ is always $q$ here.

Now we continue with the proof of Lemma \ref{lemma2.1}. 

\begin{proof} (Lemma \ref{lemma2.1}) Fix $q$, $\bfq$, $\bfa\in U_\bfq$, and $a\in U_q$. Assume, without loss of generality, that $q_1$ does not divide $q$. Then we can write $q=pd$ and $q_1=p_1d$ where the greatest common divisor of $p$ and $p_1$, denoted $(p,p_1)$, is at least $1$. 

Now \[
\sum_{s\in Z_Q^n}e(\p(s)a/q+s\cdot \bfa/\bfq)\]\[=\sum_{s_2,...,s_n\in Z_Q}\left(\sum_{s_1\in Z_Q}e(\p(s_1,...,s_n)a/q+s_1 a_1/q_1)\right)e(a_2 s_2/q_2+...+a_ns_n/q_n).\]
Let $Q_1$ be the least common multiple of $q$ and $q_1$, so that $Q_1=pp_1d$. The inner sum is a multiple of \[
\sum_{s_1\in Z_{Q_1}}e(\p(s_1,...,s_n)a/q+s_1 a_1/q_1)\]
as $Q_1|Q$ and the phase is periodic modulo $Q_1$. The sum over $s_1\in Z_{Q_1}$ can be written as a sum of $r+qt$ over $r\in Z_q$ and $t\in Z_{p_1}$, giving \[
\sum_{s_1\in Z_{Q_1}}e(\p(s_1,...,s_n)a/q+s_1 a_1/q_1)\]\[
=\sum_{r\in Z_q}\sum_{t\in Z_{p_1}}e(\p(r+qt,...,s_n)a/q+(r+q_1t) a_1/q_1)\]\[
=\sum_{r\in Z_q}e(\p(r+qt,s_2,...,s_n)a/q)e(ra_1/q_1)\left(\sum_{t\in Z_{p_1}}e(qt a_1/q_1)\right).\]
The result follows as \[
\sum_{t\in Z_{p_1}}e(qt a_1/q_1)=\sum_{t\in Z_{p_1}}e(pt a_1/p_1)=0\]
as $(a_1p,p_1)=1$.
\end{proof}

The other result we are interested is an application of Lemma \ref{lemma2.2} which concerns the number of solutions to the  equation $\p(x)=0$ over the cyclic groups $Z_q$.  

\begin{lemma}\label{lemma4.2}
Let $q\geq1$ be an integer and $\mathcal{B}(\p)>(d-1)2^d$.  Then \[
q^{1-n}\sum_{s\in Z^n_{q}}\1_{\p(s)=0 \, (q)}\<1.\]
\end{lemma}

\begin{proof}
Define\[
W_{a,q}=\sum_{s\in Z_q^n}e(\p(s)a/q)\]
so that \eq\label{4.3}
\sum_{s\in Z^n_{q}}\1_{\p(s)=0 \, (q)}=q^{-1}\sum_{a\in Z_q}W_{a,q}\ee

For general $q$ 
we have the estimate\[
|W_{a,q}|\leq q^{n-c+\eps}\]
for some $c>2$ whenever $a\in U_q$, which is Lemma \ref{lemma2.2} when $q_1=...=q_n=1$.

 Consider first   $q$ of the form $p^t$ for some prime $p$.  For each $a\in Z_q$ we can write $a=p^ra'$ where $a'\in U_{p^{t-r}}$ and $r\in\{0,1,...,t\}$, the case $r=t$ corresponding to $a=0$.  Then we have \[
W_{a,q}=\sum_{s\in Z^n_{p^t}}e(\p(s)a/p^t)=\sum_{s\in Z^n_{p^t}}e(\p(s)a'/p^{t-r})=p^{nt}/p^{n(t-r)}W_{a',p^{t-r}}\]
thus giving a bound \[
|W_{a,q}|=O(p^{rn}p^{n(t-r)-c(t-r)+\eps})=O(p^{nt-c(t-r)+\eps}).\]
 The collection of all such $a$ where $r$ is fixed is naturally equivalent to $U_{p^{t-r}}$. Then the contribution from these terms to the sum in \eqref{4.3} is  \[
\sum_{a'\in U_{p^{t-r}}}W_{a',p^{t-r}}\<p^{t-r} p^{nt-c(t-r)+\eps}=p^{nt-(1-c)(t-r)+\eps}.\]
Summing over all $r=0,...,t$ then gives\[
q^{-1}\sum_{a\in Z_q}W_{a,q}=O(q^{n-1}(1+p^{(1-c)+\eps}+p^{2(1-c)+\eps}+...+p^{t(1-c)+\eps}))\]
where it is clear that the $p^\eps$ is not necessary when $r=t$ as this corresponds to the single element $a=0\in Z_q$.

With the assumption that $c>2$ we  have that 
\[
p^{(1-c)+\eps}+p^{2(1-c)+\eps}+...+p^{t(1-c)+\eps}\leq p^{(1-c)}(1+2^{(1-c)}+2^{2(1-c)}+...)\<p^{1-c}.\]
Then \[
q^{-1}\sum_{a\in Z_q}W_{a,q}=q^{n-1}(1+O(p^{1-c}))\]
for $q$ a prime power.

For general composite $q$ we write $q=p_1^{t_1}...p_m^{t_m}$ and use the well known fact that \[
|\{s\in Z^n_q:\p(s)\equiv 0 \,(q)\}|=\prod_{i=1}^m|\{s\in Z^n_{p_i^{t_i}}:\p(s)\equiv 0 \,(p_i^{t_i})\}|\]
to get the bound \[
q^{1-n}\sum_{s\in Z^n_{q}}\1_{\p(s)=0 \, (q)}=\prod_{p|q}(1+O(p^{(1-c)+\eps})).\]
This  is  $O(1)$ as   \[
\prod_p(1+O(p^{1-c+\eps}))\]
is absolutely convergent when $c>2$. 
\end{proof}

\vspace{.2in}
%%%%%%%%%%%%%%%%%%%%%%%%%%
\section{On the approximation}

 Our goal in this section is to deduce Lemmas \ref{lemma2.6} and \ref{lemma2.115} . We handle the former first, which is an application of both Lemma \ref{lemma2.4} and Lemma \ref{lemma2.5}. 
 
 \begin{proof}(Lemma \ref{lemma2.6}) Let $l\geq0$ be a given integer and consider \[
\sum_{a\in U_q}\sum_{\bfa\in Z_q^n}F_q(a,\bfa)\left(\zeta(q(\xi-\bfa/q)-\zeta(10^l(\xi-\bfa/q))\right)\ds(N(\xi-\bfa/q))\]
for some $q\in I_l$. For each $\xi$ there is at most one $\bfa$ for which $\zeta(q(\xi-\bfa/q)-\zeta(10^l(\xi-\bfa/q))$ is nonzero, and on the support of these terms we have that $|\xi-\bfa/q|\geq 10^{-l}.$ From Lemma \ref{lemma2.5} we have the estimate \[
|\ds(N(\xi-\bfa/q))|\< (N/10^l)^{-c}\]
when $l$ is small in terms of $N$. More precisely, if  $l<\delta \log\,N$  we  have an estimate of the form\[
|\ds(N(\xi-\bfa/q))|\< N^{-\delta}\]
In turn 
\[
|\sum_{a\in U_q}\sum_{\bfa\in Z_q^n}F_q(a,\bfa)\left(\zeta(q(\xi-\bfa/q)-\zeta(10^l(\xi-\bfa/q))\right)\ds(N(\xi-\bfa/q))|\< q^{1-c}N^{-\delta}\]
uniformly in $\xi$, where we have applied  Lemma \ref{lemma2.2}. This is summable in  $q$ since $c>2$. Hence one gets  the bound \[
\sum_{l<\delta\log\,N}\sum_{q\in I_l}\sum_{a\in U_q}\sum_{\bfa\in Z_q^n}F_q(a,\bfa)\left(\zeta(q(\xi-\bfa/q)-\zeta(10^l(\xi-\bfa/q))\right)\ds(N(\xi-\bfa/q))=O(N^{-\delta}).\]

It remains to consider 
\[
\sum_{l\geq\delta\log\,N}\sum_{q\in I_l}\sum_{a\in U_q}\sum_{\bfa\in Z_q^n}F_q(a,\bfa)\left(\zeta(q(\xi-\bfa/q)-\zeta(10^l(\xi-\bfa/q))\right)\ds(N(\xi-\bfa/q)).\] 
For this we notice that \[
F_q(a,\bfa)\left(\zeta(q(\xi-\bfa/q)-\zeta(10^l(\xi-\bfa/q))\right)\ds(N(\xi-\bfa/q))\< q^{-c}\]
uniformly in $\bfa\in Z_q^n$ and $\xi\in\Pi^n$ when $q\in I_l$. Then we have the bound \[
\sum_{l\geq\delta\log\,N}\sum_{q\in I_l}\sum_{a\in U_q}q^{-c}\leq\sum_{q\geq 2^{\delta \log\,N}}q^{1-c}=O(N^{-\delta})\]
as  $c>2$ by assumption.
\end{proof}

\begin{proof} (Lemma \ref{lemma2.115})
Fix $j\geq1$ and consider\[
||\sup_{k\geq 4^{j-1}}|\mathscr{F}^{-1}(M_{2^k,j}\widehat{f})|\,||_{\ell^2}\]
for a given function $f\in\ell^2$. This is at most\eq\label{5.2}
\sum_{q\in I_j}\sum_{a\in U_q}\|\sup_{k\geq 4^{j-1}}\left|\mathscr{F}^{-1}\left(\sum_{\bfa\in Z_q^n}F_q(a,\bfa)\zeta(10^j(\cdot-\bfa/\bfq))\dsk(\cdot-\bfa/\bfq)\widehat{f}\right)\right|\,\|_{\ell^2}\ee

We use that $\zeta(10^j(\xi-\bfa/\bfq))\zeta((10^j/2)(\xi-\bfa/\bfq))=\zeta(10^j(\xi-\bfa/\bfq))$ to write\[
\sum_{\bfa\in Z_q^n}F_q(a,\bfa)\zeta(10^j(\xi-\bfa/\bfq))\dsk(\xi-\bfa/\bfq)\]\[
=\left(\sum_{a\in Z_q}\sum_{\bfa\in Z_q^n}F_q(a,\bfa)\zeta(10^j(\xi-\bfa/q))\right) \left(\sum_{\bfa\in Z_q^n}\zeta((10^j/2)(\xi-\bfa/q))\ds_k(\xi-\bfa/q)\right).\]
Because the the support of $\zeta((10^j/2)\xi)$ is contained in the cube $[-2/(5\cdot10^j),2/(5\cdot10^j)]^n$, which is in turn contained in $[-1/q,1/q]^n$ for all $q\in I_j$, we can apply (\cite{MSW}, corollary 2.1) with our continuous maximal function inequality (Lemma \ref{lemma2.7}) to get that \eqref{5.2} is bounded by \[
\sum_{q\in I_j}\sum_{a\in U_q}\|\mathscr{F}^{-1}\left(\sum_{\bfa\in Z_q^n}F_q(a,\bfa)\zeta(10^j(\cdot-\bfa/\bfq))\widehat{f}\right)\|_{\ell^2},\]
 noting that the result of \cite{MSW} has an implied constant independent of $q$. In turn this is at most \[
\sum_{q\in I_j}\sum_{a\in U_q}\left(\int_{\Pi^n}\left|\sum_{\bfa\in Z_q^n}F_q(a,\bfa)\zeta(10^j(\xi-\bfa/\bfq))\widehat{f}(\xi)\right|^2\,d\xi\right)^{1/2}\]\[
\leq\sum_{q\in I_j}\sum_{a\in U_q}\sup_{\bfa\in Z_q^n}|F_q(a,\bfa)| \,||f||_{\ell^2}.\]
The proof is completed by an application of Lemma \ref{lemma2.2}, as we have \[
\sum_{q\in I_{j}}q^{1-c}=\sum_{q=2^{j}}^{2^{j+1}}q^{1-c}=O(2^{-\delta j})\]
because of the assumption $c>2$.
\end{proof}

\vspace{.2in}
%%%%%%%%%%%%%%%%%%%%%%%%%%
\section{An $\ell^p$ inequality}

This section is devoted to the proof of Lemma \ref{lemma2.8}. The argument is another reduction to Lemma \ref{lemma2.7} following the method of \cite{MSW}. The main ingredient that we need is a result about $\ell^p\to\ell^p$ estimates involving the multipliers \[
W_q(\xi)=\sum_{a\in Z_q}\sum_{\bfa\in Z_q^n} F_q(a,\bfa)\zeta(q^2(\xi-\bfa/q)) \]
These estimates are achieved by an interpolation argument involving a result at $\ell^1$ and a result at  $\ell^2$. At $\ell^1$ we have the following lemma.
\begin{lemma}\label{lemma7.1}
We have
\[
||\mathscr{F}^{-1}(\widehat{f}W_{q})||_{\ell^1}=O(||f||_{\ell^1})
\]
uniformly in $q$.
\end{lemma}

The  $\ell^2$ estimate is stated next. 
\begin{lemma}\label{lemma7.2}
Uniformly in $q$ we have 
\eq\label{7.2}
||\mathscr{F}^{-1}(\widehat{f}W_{q})||_{\ell^2}=O(||f||_{\ell^2}).
\ee
\end{lemma}

The reduction to these estimates to the result of \cite{MSW} is essentially the same as in the previous section. Write \[
\Omega_{N,q}(\xi)=\sum_{a\in Z_q}\sum_{\bfa\in Z_q^n}F_q(a,\bfa)\zeta(q^2(\xi-\bfa/q))\ds(N(\xi-\bfa/q))\]
as \[
\left(\sum_{a\in Z_q}\sum_{\bfa\in Z_q^n}F_q(a,\bfa)\zeta(q^2(\xi-\bfa/q))\right) \left(\sum_{\bfa\in Z_q^n}\zeta(q(\xi-\bfa/q))\ds(N(\xi-\bfa/q))\right).\]
This assumes that $q\geq2$, as then $1/(5q^2)\leq 1/(10q)$ so that $\zeta(q^2(\xi-\bfa/q)\zeta(q(\xi-\bfa/q))=\zeta(q^2(\xi-\bfa/q))$. When $q=1$ we are simply dealing with \[
\Omega_{N,1}(\xi)=\zeta(\xi)\ds(N\xi)\]
and the result follows from the above mentioned result of \cite{MSW} and Lemma \ref{lemma2.7}. Then the estimate \[
||\sup_{N\geq1}|\mathscr{F}^{-1}(\Omega_{N,q}\widehat{f})|\,||_{\ell^p}\<||f||_{\ell^p}\]
is reduced to showing that  \[
||\mathscr{F}^{-1}(\widehat{f}W_{q})||_{\ell^p}=O(||f||_{\ell^p}),\]
provided of course that $p>1$ so that Lemma \ref{lemma2.7} applies. The latter estimate holds by interpolating Lemmas \ref{lemma7.1} and \ref{lemma7.2}. 

The proof of  Lemma \ref{lemma2.8}  is concluded once we have established Lemma \ref{lemma7.1} and Lemma \ref{lemma7.2}.

\begin{proof} (Lemma \ref{lemma7.1}) 
Take $f\in L^1(\Z^n)$, and we can assume that $f$ is supported on $q\Z^n+t$ for some $t\in Z^n_q$ (identifying $Z_q$ with $\{0,1,...,q-1\}$) by noticing  \[
||\mathscr{F}^{-1}\left((\sum_{t\in Z_q^n}\widehat{f_t})W_{q}\right)||_{\ell^1}\leq \sum_{t\in Z_t^n}||\mathscr{F}^{-1}(\widehat{f_t}W_{q})||_{\ell^1}\]
and \[
\sum_{t\in Z_q}||f_t||_{\ell^1}=||f||_{\ell^1}\]
where $f_t$ denotes the restriction of $f$ to the set $q\Z^n+t$. 

 Consider \[
\int_{\Pi^n}\widehat{f}(\xi)\left(\sum_{a\in Z_q}\sum_{\bfa\in Z^n_q}\zeta(\xi-\bfa/q)F_q(a,\bfa)\right)e(-y\cdot\xi)d\xi\]\[
=\sum_{a\in Z_q}\sum_{\bfa\in Z^n_q}F_q(a,\bfa)e(-y\cdot\bfa/q)\int_{\Pi^n}\widehat{f}(\xi+\bfa/q)\zeta(\xi)e(-y\cdot\xi)d\xi\]
Expand out $\widehat{f}$ by its Fourier series to get \[
\sum_{a\in Z_q}\sum_{\bfa\in Z^n_q}F_q(a,\bfa)e(-y\cdot\bfa/q)e(t\cdot\bfa/q)\sum_{l\in\Z^n}f(l)\int_{\Pi^n}\zeta(\xi)e((l-y)\cdot\xi)d\xi\]\[
=\sum_{a\in Z_q}\sum_{\bfa\in Z^n_q}F_q(a,\bfa)e((t-y)\cdot\bfa/q)\left(f*\mathscr{F}^{-1}(\zeta)\right)(y)\]
where we have used our assumption that the support of $f$ is in the set $q\Z^n+t$. Now take the $\ell^1$ norm of this expression and split the resulting sum into residue classes modulo $q$: \[
\sum_{y\in\Z^n}\left|\sum_{a\in Z_q}\sum_{\bfa\in Z^n_q}F_q(a,\bfa)e((t-y)\cdot\bfa/q)\left(f*\mathscr{F}^{-1}(\zeta)\right)(y)\right|\]\[
=\sum_{r\in Z^n_q}\sum_{z\in\Z^n}\left|\sum_{a\in Z_q}\sum_{\bfa\in Z^n_q}F_q(a,\bfa)e((t-(q z+r))\cdot\bfa/q)\left(f*\mathscr{F}^{-1}(\zeta)\right)(qz+r)\right|\]\[
=\sum_{r\in Z^n_q}\left(\sum_{z\in\Z^n}\left| \left(f*\mathscr{F}^{-1}(\zeta)\right)(qz+r)\right|\right)\left( \left|\sum_{a\in Z_q}\sum_{\bfa\in Z^n_q}F_q(a,\bfa)e((t-r)\cdot\bfa/q)\right|\right).\]

For a fixed $r$ the sum \[
\sum_{z\in\Z^n} \left|\left(f*\mathscr{F}^{-1}(\zeta)\right)(qz+r)\right|\]
is at most \eq\label{6.66}\sum_{z\in\Z^n} \sum_{l\in\Z^n}|f(l)|\,|\mathscr{F}^{-1}(\zeta)((q z+r)-l)|=
\sum_{l\in\Z^n}|f(l)|\,\left(\sum_{z\in\Z^n}|\mathscr{F}^{-1}(\zeta)((q z+r)-l)|\right).\ee
Notice that the sum \[
\sum_{z\in\Z^n}|\mathscr{F}^{-1}(\zeta)((q z+r)-l)|\]
is simply the $L^1$ norm of $\mathscr{F}^{-1}(\zeta)$ restricted to a residue class of $Z_q^n$, and in particular  is periodic in $l$ with respect to elements of $q\Z^n$. For each $l$ on the right hand side of \eqref{6.66} we  now have the bound \[
|f(l)|\sup_{t\in Z^n_q} \sum_{z\in\Z^n}|\mathscr{F}^{-1}(\zeta)(q z+t)|.\]
This is bounded by a constant multiple of $|f(l)|/{q^n}$ due to the assumption that $\zeta$ is smooth. Thus we have \[
\sum_{z\in\Z^n} \left|\left(f*\mathscr{F}^{-1}(\zeta)\right)(qz+r)\right|\<\frac{||f||_{L^1(\Z^n)}}{q^n} \]
uniformly in $r$.

Now we need to consider \[
\,\frac{1}{q^n} 
\sum_{r\in Z^n_q}\left|\sum_{a\in Z_q} \sum_{\bfa\in Z^n_q}F_q(a,\bfa)e((t-r)\cdot\bfa/q)\right|.\]
Proceed by expanding $F$ by its definition to get \[
\frac{1}{q^n} 
\sum_{r\in Z^n_q}\left| \frac{1}{q^n}\sum_{a\in Z_q}\sum_{\bfa\in Z^n_q}\sum_{s\in Z_q^n}e(\p(s)a/q+s\cdot \bfa/q)e((t-r)\cdot\bfa/q)\right|.\]
We can sum now in $a$ and $\bfa$ to arrive at \[
\frac{1}{q^n} 
\sum_{r\in Z^n_q}\left| \frac{q^{n+1}}{q^n}\sum_{s\in Z_q^n} \1_{\p(s)\equiv 0\,(q)} \1_{s\equiv (r-t)\,(q)}\right|\]
and hence we have the bound \[
q^{1-n}\sum_{s\in Z_q^n} \1_{\p(s)\equiv 0\,(q)} \]
by summing in $r$. This is $O(1)$ by Lemma \ref{lemma4.2}.
\end{proof}

\begin{proof} (Lemma \ref{lemma7.2})
By the disjointness of the supports of the terms involving the function  $\zeta$ it follows as in the proof of Lemma \ref{lemma2.115} that we have the bound \[
\sum_{a\in U_q}\sup_{\bfa\in Z_q^n}|F_q(a,\bfa)|\,||f||_{\ell^2}\]
for \eqref{7.2}, and Lemma \ref{lemma2.2} then gives the bound $q^{1-c}$. This is clearly $O(1)$ independent of $q$ under the assumption that $c>2$.
\end{proof}

\vspace{.2in}
%%%%%%%%%%%%%%%%%%%%%%%%%%
\section{$\ell^2$ estimates}

Here we prove the error term estimates in Lemma \ref{lemma2.10}, Lemma \ref{lemma2.11},  and Lemma \ref{lemma2.12}. The proofs are given in order. Again we use $\dsk(\xi)$ to mean $\ds(2^k\xi)$

\begin{proof} (Lemma \ref{lemma2.10})
Recall that \[
E^{(1)}_{2^k,j}(\xi)=\sum_{l=0}^{j-1}\sum_{q\in I_l}\sum_{a\in U_q}\sum_{\bfa\in Z^n_q}F_q(a,\bfa)\left(\zeta(Q_j^2(\xi-\bfa/q))-\zeta(10^l(\xi-\bfa/q)\right)\dsk((\xi-\bfa/q)).\]
With $f\in \ell^2$ we have that \[
||\sup_{N=2^k;\,k\geq4^{j-1}}|\mathscr{F}^{-1}(E^{(1)}_{N,j}\widehat{f})|\,||_{\ell^2}\]
is bounded by \[
\sum_{N=2^k;\,k\geq4^{j-1}}||\mathscr{F}^{-1}(E^{(1)}_{N,j}\widehat{f})|\,||_{\ell^2}\leq\sum_{N=2^k;\,k\geq4^{j-1}}||(E^{(1)}_{N,j}||_{L^\infty(\Pi^n)}||f||_{\ell^2}.\]
using Plancherel's Theorem. In turn this is at most \[
||f||_{\ell^2}\sum_{N=2^k;\,k\geq4^{j-1}}\sum_{l=0}^{j-1}\sum_{q\in I_l}\sum_{a\in U_q}\left(\sup_{\bfa\in Z^n_q}|F_q(a,\bfa)|\right)\]\[\times\,||\left(\zeta(Q_j^2(\xi-\bfa/q))-\zeta(10^l(\xi-\bfa/q))\right)\dsk(\xi-\bfa/q)||_{L^\infty(\Pi^n)}.\]

On the support of $\zeta(Q_j^2(\xi-\bfa/q))-\zeta(10^l(\xi-\bfa/q)$ we have that $\dsk$ is bounded above in terms of $Q_j^{2c}/2^{ck}$ by Lemma \ref{lemma2.5}.  Notice that trivially  $Q_j\leq 2^{j2^j}$. Then the sum in $k$ has summands which are crudely bounded as \[
\sum_{l=0}^{j-1}\sum_{q\in I_l}\sum_{a\in U_q}Q_j^{2c}/2^{ck}\leq 2^{2j}2^{cj2^{j+1}}2^{-ck}.\]
Summing in $k\geq 4^{j-1}$ gives a bound of $O(2^{-\delta j})$.\end{proof}

\begin{proof} (Lemma \ref{lemma2.11})
Write\[
E^{(2)}_{N,j}(\xi)=\sum_{q|Q_j;\,q\geq 2^j}\sum_{a\in U_q}\sum_{\bfa\in Z^n_q}F_q(a,\bfa)\zeta(Q_j^2(\xi-\bfa/q))\ds(N(\xi-\bfa/q))\]\[
=\sum_{q|Q_j;\,q\geq 2^j}\sum_{a\in U_q} \mathcal{T}_{a,q}(\xi).\]
Applying (\cite{MSW}, \S2, Corollary 2.1) again gives that \[
||\sup_{N=2^k;\,k\geq4^{j-1}}|\mathscr{F}^{-1}(\mathcal{T}_{a,q}\widehat{f})|\,||_{\ell^2}\]
is bounded by the $L^2(\Pi^n)$ norm of \[
\sum_{\bfa\in Z^n_q}F_q(a,\bfa)\zeta(Q_j^2(\xi-\bfa/q))\widehat{f}(\xi),\]
which is at most \[
\left(\sup_{\bfa\in Z^n_q}|F_q(a,\bfa)|\right)||f||_{\ell^2}\< q^{-c}||f||_{\ell^2}\]
by Lemma \ref{lemma2.2}. Finally we have \[
\sum_{q|Q_j;\,q\geq 2^j}\sum_{a\in U_q}q^{-c}=\sum_{q|Q_j;\,q\geq 2^j}q^{1-c}=O(2^{-\delta j})\]
when $c>2$.
\end{proof}

\begin{proof} (Lemma \ref{lemma2.12})
We have that \[
||\sup_{k\geq 4^{j-1}}|K_k*f-\mathscr{F}^{-1}(\sum_{l=0}^{j-1}M_{2^k,l}\widehat{f})|\,||_{\ell^2}\]
is bounded above by \[
||\sup_{k\geq 4^{j-1}}|K_k*f-\mathscr{F}^{-1}(\sum_{l=0}^{\infty}M_{2^k,l}\widehat{f})|\,||_{\ell^2}\]\[
+||\sup_{k\geq 4^{j-1}}|\mathscr{F}^{-1}(\sum_{l=j}^{\infty}M_{2^k,l}\widehat{f})|\,||_{\ell^2}\]
The first of these terms is bounded by \[
\sum_{k\geq 4^{j-1}}\sup_{\xi}|\widehat{\omega_{2^k}}(\xi)-\sum_{l=0}^{\infty}M_{2^k,l}(\xi)|\,||f||_{\ell^2}\]
and the desired bound follows from  Lemma \ref{lemma2.6}, as\[
\sum_{k\geq 4^{j-1}}2^{-\delta k}=O(2^{-\delta j}).\]
For the remaining term we have 
\[
||\sup_{k\geq 4^{j-1}}\mathscr{F}^{-1}(\sum_{l=j}^{\infty}M_{2^k,l}\widehat{f})|\,||_{\ell^2}\leq \sum_{l=j}^{\infty}||\sup_{k\geq 4^{j-1}}\mathscr{F}^{-1}(M_{2^k,l}\widehat{f})|\,||_{\ell^2}\]\[=
 \sum_{l=j}^\infty O(2^{-\delta l})=O(2^{-\delta j})\]
 by appealing to the method of Lemma \ref{lemma2.115}.
This completes the proof. 
\end{proof}

\vspace{.2in}
%%%%%%%%%%%%%%%%%%%%%%%%%%
\section{The partial maximal inequality}

The proof of Lemma \ref{lemma2.9} follows the outline given by Bourgain for proving partial maximal function estimates. The argument appears in (\cite{Bou}, section 7), and is also carried out in (\cite{Nair}, section 10) with more detail.  

Let $G=Z_J$ for some large integer $J$ and endow $G$ with the normalized counting measure. Identify $G$ with the set $\{0,1,...,J-1\}$. Because $\omega_N$ is a positive kernel, Lemma \ref{lemma2.9}  follows from the estimate 
\eq\label{8.1}
||\sup_{k_0\leq k\leq 2k_0}|f*K_k|\,||_{L^p(G)}\< \log(k_0) ||f||_{L^p(G)}.
\ee
The reason for the formulation with the compact group $G$ is to apply a result from \cite{Stein} which implies that \eqref{8.1} holds true if we can show the weaker inequality
\[
||\sup_{k_0\leq k\leq 2k_0}|f*K_k|\,||_{L^1(G)}\< \log(k_0) ||f||_{L^p(G)}.
\]
Written in the dual form this becomes 
\[
||\sum_{k=k_0}^ {2k_0}g_k*K_k||_{\ell^u}\< \log(k_0) 
\]
where $u=p/(p-1)$ and $g_k$ is any collection of nonnegative functions with $\sum_{k}g_k\leq 1.$ Also, as \eqref{8.1} weakens as $p$ increases, we can assume that $u$ is an integer. 

The argument of Bourgain for dealing with partial maximal functions mentioned above is a very general reduction by Fourier analytic methods to an $L^2$ result, and it is only in this result where we need to make modifications of that argument.  Set $L_k=K_{mk}$ for some $m=C_1 \log k_0$, $C_1$ to be chosen later.  Lemma \ref{lemma2.9} follows once we have established  our  last result, which in turn  completes the proof of Theorem \ref{thm11}. 

\begin{lemma}
Let $k_0\leq k_1<...<k_u<2k_0$. Then we have the estimate
\eq\label{8.4}
||[g_{k_2}*L_{k_2})...(g_{k_u}*L_{k_u})]*(L_{k_1}-L_{k_0})]||_{L^2(G)}\leq k_0^{-u}
\ee
for any functions $g_{k_2},...,g_{k_u}\geq0$ satisfying \[
\sum_i g_{k_i}\leq 1.\]
\end{lemma}

\begin{proof}
Set $g_{k_i}=g_i$, and $N_i=2^{mk_i}$ for each $i$. For each $i$ we select  integers $l_i=\alpha_i \log(k_0)$ below for an increasing sequence $\alpha_i$, and  also fix  $D$ with $D =k^{C_2}$, $C_2$ also to be chosen below. 

Define\[
\Omega_i(\xi)=\sum_{l\leq l_i}\sum_{q\in I_l}\sum_{a\in U_q}\sum_{\bfa\in Z_q^n}F_q(a,\bfa)\ds(N_i(\xi-\bfa/q))\zeta(10^l(\xi-\bfa/q))\zeta(N_i/D(\xi-\bfa/q)).\]
Estimate this by removing the $\zeta$ term with $D$, and extending the sum in $l$ to $\infty$:\[
\widehat{\omega_{N_i}}(\xi)-\Omega_i(\xi)=\widehat{\omega_{N_i}}(\xi)-\sum_{l=1}^\infty\sum_{q\in I_l}\sum_{a\in U_q}\sum_{\bfa\in Z_q^n}F_q(a,\bfa)\ds(N_i(\xi-\bfa/q))\zeta(10^l(\xi-\bfa/q))\] 
\[-\sum_{l\leq l_i}\sum_{q\in I_l}\sum_{a\in U_q}\sum_{\bfa\in Z_q^n}F_q(a,\bfa)\ds(N_i(\xi-\bfa/q))\zeta(10^l(\xi-\bfa/q))(1-\zeta(N_i/D(\xi-\bfa/q)))\] 
\[+\sum_{l> l_i}\sum_{q\in I_l}\sum_{a\in U_q}\sum_{\bfa\in Z_q^n}F_q(a,\bfa)\ds(N_i(\xi-\bfa/q))\zeta(10^l(\xi-\bfa/q))|\] 
\eq\label{8.5}\<N_i^{-\delta}+D^{-c}+ 2^{-l_i\delta},\ee
as \[
\widehat{\omega_{N_i}}(\xi)-\sum_{l=1}^\infty\sum_{q\in I_l}\sum_{a\in U_q}\sum_{\bfa\in Z_q^n}F_q(a,\bfa)\ds(N_i(\xi-\bfa/q))\zeta(10^l(\xi-\bfa/q))=O(N_i^{-\delta})\]
by Lemma \ref{lemma2.6};  \[
\sum_{l\leq l_i}\sum_{q\in I_l}\sum_{a\in U_q}\sum_{\bfa\in Z_q^n}F_q(a,\bfa)\ds(N_i(\xi-\bfa/q))\zeta(10^l(\xi-\bfa/q))(1-\zeta(N_i/D(\xi-\bfa/q)))\]\[
\leq\sum_{l\leq l_i}\sum_{q\in I_l}\sup_{\bfa\in U_q^n}q|F_q(a,\bfa)|\,O((N_iD/N_i)^{-c})=O(D^{-c})\]
by Lemmas \ref{lemma2.2} and \ref{lemma2.5}; and \[
\sum_{l> l_i}\sum_{q\in I_l}\sum_{a\in U_q}\sum_{\bfa\in Z_q^n}F_q(a,\bfa)\ds(N_i(\xi-\bfa/q))\zeta(10^l(\xi-\bfa/q))=O(2^{-{\delta l_i}})\] 
holds uniformly in $\xi$.

We have that \[
||\mathscr{F}^{-1}(\Omega_i)||_{\ell^1} \< 2^{(n+2)l_i}\]
by treating the summands independently. 
In turn we then have that \[
|\mathscr{F}^{-1}(\Omega_ig_i)|\<2^{(n+2)l_i}\]
holds uniformly. We also get that   \[
|| (g_i*L_{k_i})-\mathscr{F}^{-1}[\Omega_i\widehat{g_i}]||_{L^2(G)}\< N_i^{-\delta}+ 2^{-l_i\delta}+D^{-c}\]
as a consequence of \eqref{8.5}.

The left hand side of \eqref{8.4} is estimated in the following manner. Begin by writing\[
||[(g_{2}*L_{k_2})...(g_{u}*L_{k_u})]*(L_{k_1}-L_{k_0})||_{L^2(G)}\]\[
\leq||[(g_{2}*L_{k_2}-\mathscr{F}^{-1}(\Omega_2\widehat{g_{2}}))...(g_{u}*L_{k_u})]*(L_{k_1}-L_{k_0})||_{L^2(G)}\]\[
+||\mathscr{F}^{-1}(\Omega_2\widehat{g_{2}})...(g_{u}*L_{k_u})]*(L_{k_1}-L_{k_0})||_{L^2(G)}.\]
The first term is at most \[
||(g_{2}*L_{k_2}-\mathscr{F}^{-1}(\Omega_2\widehat{g_{2}})||_{L^2(G)}.\]
Now repeat this process for the remaining term leading to an overall bound of the form
\[
||(g_2*L_{k_2})-\mathscr{F}^{-1}[\Omega_2\widehat{g_2}]||_{L^2(G)}\]\[
+||\mathscr{F}^{-1}[\Omega_2\widehat{g_2}]||_{L^\infty(G)}||(g_3*L_{k_3})-\mathscr{F}^{-1}[\Omega_3\widehat{g_3}||_{L^2(G)}+...\]
\[
+||\mathscr{F}^{-1}[\Omega_2\widehat{g_2}]||_{L^\infty(G)}...||\mathscr{F}^{-1}[\Omega_{u-1}\widehat{g_{u-1}}]||_{L^\infty(G)}||(g_u*L_{k_u})-\mathscr{F}^{-1}[\Omega_u\widehat{g_u}||_{L^2(G)}\]
\[
+||[\mathscr{F}^{-1}(\widehat{g_2}\Omega_2)...\mathscr{F}^{-1}(\widehat{g_u}\Omega_u)]*(L_{k_1}-L_{k_0})||_{L^2(G)}.
\]
In turn this is bounded above by \[
\sum_{i=2}^u(2^{n+2})^{l_2+...+l_{i-1}}(N_i^{-\delta}+ 2^{-l_{i}\delta}+D^{-c})
+||[\mathscr{F}^{-1}(\widehat{g_2}\Omega_2)...\mathscr{F}^{-1}(\widehat{g_q}\Omega_q)]*(L_{k_1}-L_{k_0})||_{L^2(G)}.\]
Iteratively choosing the $\alpha_i$   gives that the first term, the sum,  is bounded above by \eq\label{8e1}
k_0^{-u}/10 +k_0^C(N_i^{-\delta}+D^{-c}).\ee

Now \[
||[\mathscr{F}^{-1}(\widehat{g_2}\Omega_2)...\mathscr{F}^{-1}(\widehat{g_u}\Omega_u)]*(L_{k_1}-L_{k_0})||_{L^2(G)}\]
is at most \eq\label{8.5}
||[\mathscr{F}^{-1}(\widehat{g_2}\Omega_2)...\mathscr{F}^{-1}(\widehat{g_u}\Omega_u)]||_{L^2(G)}\sup_{\xi\in \Gamma}|\widehat{L_{k_1}}(\xi)-\widehat{L_{k_0}}(\xi)|.\ee
Here  $\Gamma$ is the sumset of the sets $\Gamma_i\subset\Pi^n$ which are $D/(5N_i)$-neighbourhoods of the sets\[
\{a/q:a\in U_q;\, q\leq 2^{l_i}\}.\]
Thus $\Gamma$ is comparable to a $D/N_2$-neighbourhood of the set\[
\{a/q:a\in U_q;\, q\leq 2^{u\,l_u}\}.\]
This of course follows by considering the support of the Fourier transform of $\mathscr{F}^{-1}(\widehat{g_2}\Omega_2)...\mathscr{F}^{-1}(\widehat{g_u}\Omega_u)$.

Now estimate \[
\widehat{L_{k_1}}(\xi)-\widehat{L_{k_0}}(\xi)\]
by \[
\sum_{l=1}^\infty\sum_{q\in I_l}\sum_{a\in U_q}\sum_{\bfa\in Z_q^n} F_q(a,\bfa)\zeta(10^l(\xi-\bfa/q))\ds(N_1(\xi-\bfa/q))\]
\[-\sum_{l=1}^\infty\sum_{q\in I_l}\sum_{a\in U_q}\sum_{\bfa\in Z_q^n} F_q(a,\bfa)\zeta(10^l(\xi-\bfa/q))\ds(N_0(\xi-\bfa/q))\]\[
+O(N_0^{-\delta}).\]
Use that $|1-\ds(\xi)|\< |\xi|$ to get that  $|\ds(N_1(\xi-\bfa/q))-\ds(N_0(\xi-\bfa/q))|\< N_1D/N_2$ on $\Gamma$, and then we have the estimate \[
|\widehat{L_{k_1}}(\xi)-\widehat{L_{k_0}}(\xi)|\< \sum_{q=1}^{\infty}q^{1-c}N_1D/N_2+N_0^{-\delta}\< N_1D/N_2+N_0^{-\delta}.\]

We have that \eqref{8.5} is bounded by  
\[k_0^C(D(N_1/N_2)+N_0^{-\delta}).\]
which is at most 
\eq\label{8e2}k_0^C(D2^{-m}+2^{-\delta m k_0}).\ee 
Then \eqref{8e1} and \eqref{8e2} add to at most
\[k_0^{-u}/10 +k_0^C(N_i^{-\delta}+D^{-c})+k_0^C(D2^{-m}+2^{-\delta m k_0}).\]
Choose $C_2$ so that $k_0^CD^{-c}$ is at most $k_0^{-u}/10$, and then we choose $C_1$ so that \[
k_0^C(D2^{-m}+2^{-\delta m k_0})\leq k_0^{-u}/10\]
and the lemma is proven.
\end{proof}

\vskip0.2in
\noindent \author{\textsc{Brian Cook}}\\
Department of Mathematics \\
Kent State University\\
Kent, OH, USA\\
Electronic address: \texttt{briancookmath@gmail.com}

\end{document}